\documentclass[a4paper,12pt]{amsart}

\usepackage{amssymb}
\usepackage{hyperref}

\newcommand\CC{\mathbb{C}}
\newcommand\NN{\mathbb{N}}

\newcommand\go{G^{(0)}}
\newcommand\ho{H^{(0)}}
\newcommand\lo{L^{(0)}}

\newcommand\op{\operatorname{op}}
\newcommand\lsp{\operatorname{span}}
\newcommand\supp{\operatorname{supp}}
\newcommand\KP{\operatorname{KP}}
\newcommand\sing{\operatorname{sing}}
\newcommand\fin{\operatorname{fin}}

\newcommand{\cs}{\ensuremath{C^{*}}}

\theoremstyle{plain}
\newtheorem{theorem}{Theorem}[section]
\newtheorem{cor}[theorem]{Corollary}
\newtheorem{lemma}[theorem]{Lemma}
\newtheorem{prop}[theorem]{Proposition}
\theoremstyle{remark}
\newtheorem{rmk}[theorem]{Remark}
\newtheorem{example}[theorem]{Example}
\theoremstyle{definition}
\newtheorem{dfn}[theorem]{Definition}

\numberwithin{equation}{section}

\title{Equivalent groupoids have Morita equivalent Steinberg algebras}

\author{Lisa Orloff Clark}
\address{Lisa Orloff Clark\\
    Department of Mathematics and Statistics\\
    University of Otago\\
    PO Box 56\\
    Dunedin 9054\\
    New Zealand}
\email{lclark@maths.otago.ac.nz}

\author{Aidan Sims}
\address{Aidan  Sims\\
    School of Mathematics and Applied Statistics\\
    University of Wollongong\\
    NSW 2522\\
    Australia}
\email{asims@uow.edu.au}

\thanks{This research was supported by the Australian Research Council.}


\date{\today}

\begin{document}
\begin{abstract}
Let $G$ and $H$ be Hausdorff ample groupoids and let $R$ be a commutative unital ring. We show that if $G$
and $H$ are equivalent in the sense of Muhly-Renault-Williams, then the associated
Steinberg algebras of locally constant $R$-valued functions with compact support 
are Morita equivalent. We
deduce that collapsing a ``collapsible subgraph" of a directed graph in the sense of
Crisp and Gow does not change the Morita-equivalence class of the associated Leavitt path
$R$-algebra, and therefore a number of graphical constructions which yield Morita
equivalent $C^*$-algebras also yield Morita equivalent Leavitt path algebras.
\end{abstract}

\maketitle

\section{Introduction}

Two groupoids $G$ and $H$ are equivalent if they act freely and properly on the left and
right (respectively) of a space $Z$ in such a way that the quotient of $Z$ by the action of $G$ is
homeomorphic to the unit space of $H$ and vice versa. It was shown in \cite{MRW87} that
if second-countable, locally compact, Hausdorff groupoids $G$ and $H$ are equivalent,
then the associated full $C^*$-algebras are Morita equivalent. This result descends to
reduced $C^*$-algebras, and also persists for groupoids which are locally Hausdorff (see
\cite{SW}). The proof of this statement in \cite{SW} proceeds by constructing a linking
groupoid $L$ from copies of $G, H, Z$ and the opposite space $Z^{\op}$ so that the
groupoid $C^*$-algebra of $L$ is a linking algebra for a $C^*(G)$--$C^*(H)$-imprimitivity
bimodule.

Given a 
Hausdorff ample groupoid $G$ and a commutative unital ring $R$,
we consider the convolution $R$-algebra $A_R(G)$ of locally constant functions  with compact support
from $G$ to $R$.  We call $A_R(G)$ the \emph{Steinberg algebra} associated to $G$.
These algebras were introduced in \cite{Steinberg} as a model for discrete inverse semigroup algebras.
In the situation where $R = \CC$,  $A_{\CC}(R)$ is a dense subalgebra of $\cs(G)$.  
Complex Steinberg algebras also include complex Kumjian-Pask algebras \cite{ACaHR} and
hence complex Leavitt path algebras.   
Uniqueness theorems and simplicity criteria for complex Steinberg algebras
are established in  \cite{BCFS} and \cite{CFST}. 
These results indicate that 
the groupoid approach is a good 
unifying framework for understanding the striking similarities between the theory of graph
$C^*$-algebras and the theory of Leavitt path algebras, which have attracted a lot of attention in recent years.

In this paper we present further evidence for this viewpoint. 
First we show that all Leavitt path $R$-algebras
can be realised as Steinberg algebras (see example~\ref{ex:lpa}).
Next we show that if $G$ and $H$
are  Hausdorff ample groupoids, and if $Z$ is a
$G$--$H$ equivalence, then the linking-groupoid construction of \cite{SW} yields another
 Hausdorff ample groupoid $L$.  We then
show that the Steinberg algebra $A_R(L)$ is, in the
appropriate sense, a linking algebra for a surjective Morita context between $A_R(G)$ and
$A_R(H)$, and hence that these two algebras are Morita equivalent.

We conclude by applying our result to the ``collapsible subgraph" construction of Crisp
and Gow \cite{CG}. They identify a specific type of subgraph $T$ of a countable directed
graph $E$ and a collapsing process that yields a new graph $F$ with vertices $E^0
\setminus T^0$, and show that $C^*(E)$ and $C^*(F)$ are Morita equivalent by realising
one as a full corner of the other. We show that this is an instance of the
Morita-equivalence theorem of \cite{MRW87} using the notion of an abstract transversal of
the groupoid of $E$ (see \cite[Example~2.7]{MRW87}). We conclude that for
arbitrary directed graphs $E$ and commutative unital rings $R$, Crisp and Gow's collapsible
subgraph construction yields Morita equivalent Leavitt path $R$-algebras $L_R(E)$ and
$L_R(F)$.

\section{Preliminaries}

A groupoid is a small category in which every morphism has an inverse. Given a groupoid
$G$, we write $r(\alpha)$ and $s(\alpha)$ for the \emph{range} and \emph{source} of
$\alpha \in G$. We call the common image of $r$ and $s$ the \emph{unit space} of $G$ and
denote it $ \go$. We identify the set of identity morphisms of $G$ with $\go$.

An \emph{\'etale} groupoid is a groupoid $G$ endowed with a topology so that composition
and inversion are continuous, and the source map $s$ is a local homeomorphism.  In this
case, $r$ is also a local homeomorphism and there is a basis of \emph{open bisections}; that is, a basis
of sets $B\subseteq G$ such that $s$ and $r$ restricted to $B$ are homeomorphisms.
We say a  groupoid is \emph{ample} if it has a basis of compact open bisections.  Note
that a Hausdorff groupoid is ample if and only if it is locally compact, Hausdorff and \'etale 
and its unit space is totally disconnected (see \cite[Lemma~2.1]{CFST}).  See \cite{Paterson} for 
more details on \'etale and ample groupoids.

We use the notational convention that if $A,B$ are subsets of a groupoid $G$, then
\[
AB := \{\alpha\beta : \alpha \in A, \beta \in B, s(\alpha) = r(\beta)\}.
\]
If $A = \{\alpha\}$, then we write $\alpha B$ for $\{\alpha\}B$.
The \emph{orbit} of a unit $x \in \go$ is the set \[[x]:=s(xG) = r(Gx) \subseteq \go.\] An
(algebraic) isomorphism $\Phi : G \to H$ of groupoids is a bijection from $G$ to $H$ that
carries units to units, preserves the range and source maps and satisfies
$\Phi(\alpha\beta) = \Phi(\alpha)\Phi(\beta)$ whenever $\alpha$ and $\beta$ are
composable in $G$. Uniqueness of inverses implies that $\Phi(\alpha^{-1}) =
\Phi(\alpha)^{-1}$. If $G$ and $H$ are topological groupoids then an isomorphism $\Phi :
G \to H$ is an algebraic isomorphism that is also a homeomorphism.

The next example demonstrates how groupoids are useful in the study of graph algebras.

\begin{example}
\label{ex:graphgroupoid} Let $E = (E^0, E^1, r_E, s_E)$ be an arbitrary directed 
graph.\footnote{To avoid confusion, we adopt the
convention that an unadorned $r$ or $s$ will always denote the range or source map in a
groupoid, and the range and source maps associated to a graph $E$ will always be
decorated with a subscript $E$.} We denote the infinite-path space by $E^{\infty}$ and
the finite-path space by $E^*$. We use the convention that a path $x$ is a sequence of
edges $x_i$ in which each $s_E(x_i)=r_E(x_{i+1})$ and we write $|x|$ for the length of $x$.  
A \emph{source} in $E$ is a vertex $v$
such that $r_E^{-1}(v) = \emptyset$, and an infinite receiver is a vertex $v$ such that
$r_E^{-1}(v)$ is infinite.

The following construction of a groupoid $G_E$ from a graph $E$ can be found in \cite{Pat2002}.
This generalises the construction in \cite{KPRR}. 
Unlike \cite{KPRR} and \cite{Pat2002}, we do not require our graphs to be countable.  
More general versions are described
in \cite{FMY,KP, RSWY,Y}.

Define
\[
X := E^\infty \cup \{\mu \in E^* \mid s_E(\mu)\text{ is a source}\} \cup 
	\{\mu \in E^* \mid s_E(\mu)\text{ is an infinite receiver}\}.
\]
Let
\[G_E := \{(\alpha x, |\alpha| - |\beta|, \beta x) \mid \alpha,\beta \in E^*,
    x \in X, s_E(\alpha) = s_E(\beta) = r_E(x)\}
\]
We view each $(x,k,y) \in G_E$ as a morphism with range $x$ and source $y$. The formulas
\[
    (x,k,y)(y,l,z) := (x, k+l, z) \quad\text{ and }\quad (x,k,y)^{-1} := (y,-k,x)
\]
define composition and inverse maps on $G_E$ making it a groupoid with 
\[\go_E =\{(x, 0, x) : x \in X\} \text{ which we identify with $X$}.\]
Next, we describe a topology on $G$. For $\mu \in E^*$, the cylinder set $Z(\mu) \subseteq X$
is the set
\[
    Z(\mu):= \{\mu x \mid x \in X, s_E(\mu)=r_E(x)\}.
\]
For $\mu \in E^*$ and a finite $F \subseteq r_E^{-1}(s_E(\mu))$, define
\[
Z(\mu\setminus F) := Z(\mu) \cap \Big(\bigcup_{\alpha \in F} Z(\mu\alpha)\Big)^c.
\]
The sets $Z(\mu \setminus F)$ are a basis of compact open sets for a locally
compact, Hausdorff topology on $X = \go_E$ (see \cite[Theorem~2.1]{Webster:xx11}).

For $\mu,\nu \in E^*$ with $s_E(\mu) = s_E(\nu)$, and for a finite $F \subseteq E^*$ such
that $s_E(\mu)=r_E(\alpha)$ for all $\alpha \in F$, we define
\[
Z(\mu,\nu) := \{(\mu x, |\mu| - |\nu|, \nu x) : x \in X, s_E(\mu)= r_E(x)\},
\]
and then
\[
Z((\mu,\nu) \setminus F) := Z(\mu,\nu) \cap
            \Big(\bigcup_{\alpha \in F} Z(\mu\alpha,\nu\alpha)\Big)^c.
\]
The $Z((\mu,\nu) \setminus F)$ form a basis of compact open sets for a locally
compact Hausdorff topology on $G_E$ under which it is \'etale.  Hence, $G_E$ is 
a Hausdorff ample groupoid.   We will come back to this
example in Example~\ref{ex:lpa} and again in Section~\ref{SectionGraphApp}.
\end{example}

\section{Steinberg algebras over commutative rings with 1}

Throughout this section, $R$ denotes a commutative unital ring, $\Gamma$ denotes a
discrete group, $G$ denotes a  Hausdorff ample groupoid, and $c$ denotes a continuous 
homomorphism from $G$ to $\Gamma$; that is, $c : G \to \Gamma$ is a continuous groupoid 
\emph{cocycle}. The Steinberg algebra $A(G)$ of $G$, introduced in \cite{Steinberg}
\footnote{Steinberg's notation is $RG$, but we continue to use the notation of 
\cite{BCFS, CFST}.} is the $R$-algebra of locally constant $R$-valued functions on $G$ 
with compact support, where addition is pointwise and multiplication
is given by convolution 
\[
	(f * g)(\gamma) = \sum_{\alpha\beta = \gamma} f(\alpha)g(\beta).
\]  
It is useful to note that 
\[
	A_R(G) = \lsp\{1_U : U \text{ is a compact open bisection of } G\} \subseteq R^G,
\]
where $1_U$ denotes the characteristic function on $U$
(see \cite[Proposition~4.3]{Steinberg}). We have
\[
 1_U * 1_V = 1_{UV}
\] for compact open bisections $U$ and $V$
(see \cite[Proposition~4.5(3)]{Steinberg}).

\begin{lemma}\label{lem:grading}
Suppose that $R$ is a commutative unital ring, $G$ is a Hausdorff ample groupoid and 
$c:G \to \Gamma$ is a continuous cocycle. The subsets
\[
    A_R(G)_n:= \{f \in A_R(G) : \supp(f)\subseteq c^{-1}(n)\}
\]
for $n \in \Gamma$ form a $\Gamma$-grading of $A_R(G)$.
\end{lemma}
\begin{proof}
We must show that:
\begin{enumerate} \item\label{it1:grading} $A_R(G) = \bigoplus_{n \in \Gamma} A_R(G)_n$
    as an $R$-module; and
\item \label{it2:grading} if $f \in A_R(G)_n$ and $g \in A_R(G)_m$ then $f*g \in
    A_R(G)_{n+m}$.
\end{enumerate}
Fix a compact
open bisection $U \subseteq G$.
For~(\ref{it1:grading}), it suffices to show that the indicator function $1_U$ 
belongs to $\bigoplus_{n \in \Gamma} A_R(G)_n$.  For $n \in \Gamma$, let $V_n := U \cap c^{-1}(n)$. Since
the $c^{-1}(n)$ are disjoint clopen sets and $U$ is compact open, the $V_n$ are disjoint compact
open subsets of $U$. Further, since $U$ is compact, only finitely many $V_n$ are nonempty, and
then $1_U = \sum_{V_n \not= \emptyset} 1_{V_n} \in \bigoplus_{n \in \Gamma} A_R(G)_n$.

For~(\ref{it2:grading}), suppose that $f \in A_R(G)_n$ and $g \in A_R(G)_m$. For
$\gamma \in G$ we have $(f * g)(\gamma) = \sum_{\alpha\beta = \gamma} f(\alpha)g(\beta)$,
and so \[\supp(f*g) \subseteq \supp(f)\supp(g) \subseteq c^{-1}(n)c^{-1}(m) \subseteq
c^{-1}(n+m).\]  Therefore $f*g \in A_R(G)_{n+m}$.
\end{proof}

\begin{example} 
\label{ex:lpa}
Every Leavitt path algebra is a Steinberg algebra.
To see this, let $E$ be an arbitrary directed graph, $G_E$ the groupoid of
Example~\ref{ex:graphgroupoid} and $R$ a commutative unital ring.  
We show that the Leavitt path algebra $L_R(E)$ is isomorphic
to $A_R(G_E)$. It is routine to check that the indicator functions $q_v
:= 1_{Z(v)}$, $v \in E^0$ are mutually orthogonal idempotents, and that the indicator
functions $t_e := 1_{Z(e, s(e))}$ and $t_{e^*} = 1_{Z(s(e), e)}$ constitute a Leavitt
$E$-family as in \cite[Definition~2.4]{Tomforde}. So the universal property of $L_R(E)$ gives a 
homomorphism $\pi : L_R(E) \to
A_R(G_E)$ satisfying $\pi(p_v) = q_v$, $\pi(s_e) = t_e$ and $\pi(s_{e^*}) = t_{e^*}$. An
application of the graded uniqueness theorem \cite[Theorem~4.8]{Tomforde} shows that this homomorphism
is injective. To see that it is surjective, observe that each $1_{Z((\mu,\nu) \setminus
F)} = t_\mu t_{\nu^*} - \sum_{\alpha \in F} t_{\mu\alpha} t_{(\nu\alpha)^*}$ belongs to
the range of $\pi$. Fix a compact open $U$. This $U$ can be written as a union of basic open
sets (because it is open), and therefore as a finite union of basic open sets (because it
is compact); say $U = \bigcup_{(\mu,\nu, F) \in \mathcal{F}} Z((\mu,\nu) \setminus F)$.
We claim that $U$ can be written as a disjoint union of basic open sets.  
By the inclusion-exclusion principle,
\[
U = \bigsqcup_{\emptyset \not= \mathcal{G} \subseteq \mathcal{F}}
    \bigg(\Big(\bigcap_{(\mu,\nu,P) \in \mathcal{G}} Z((\mu,\nu)\setminus P)\Big)
        \setminus \Big(\bigcup_{(\eta,\zeta,Q) \in \mathcal{F} \setminus \mathcal{G}}
            Z((\eta,\zeta)\setminus Q)\Big)\bigg)
\]
For any $\mu,\nu,\alpha,\beta \in E^*$ with $s(\mu) = s(\nu)$ and $s(\alpha) = s(\beta)$,
we have
\begin{align*}
Z(\mu,\nu) \cap Z(\alpha,\beta)
    &= \begin{cases}
        Z(\alpha,\beta) &\text{ if $\alpha = \mu\tau$ and $\beta = \nu\tau$} \\
        Z(\mu,\nu) &\text{ if $\mu = \alpha\tau$ and $\nu = \beta\tau$} \\
        \emptyset &\text{ otherwise,}
    \end{cases} \intertext{and}
Z(\mu,\nu) \setminus Z(\alpha,\beta)
    &= \begin{cases}
        Z((\mu,\nu) \setminus \{\tau\}) &\text{ if $\alpha = \mu\tau$ and $\beta = \nu\tau$} \\
        \emptyset &\text{ otherwise.}
    \end{cases}
\end{align*}
Using this, de Morgan's laws and distributivity of intersection and union, it is routine
to check that every set of the form $\bigcap_{(\mu,\nu,P) \in \mathcal{G}} Z((\mu,\nu)
\setminus P) \setminus (\bigcup_{(\eta,\zeta,Q) \in \mathcal{H}} Z((\eta,\zeta), Q))$ with
$\mathcal{G, H}$ finite and $\mathcal{G}$ nonempty can be written as a finite disjoint
union of basic open sets. Hence $U$ can be written as a finite disjoint union of basic
open sets as claimed. Thus $1_U$ is a finite sum of indicator functions of basic open sets, and
therefore belongs to the range of $\pi$. That is, $\pi$ is an isomorphism of $L_R(E)$
onto $A_R(G_E)$ as required.
\end{example}

\begin{rmk}
If $\Lambda$ is a row-finite $k$-graph with no sources and $G_{\Lambda}$ is the
associated groupoid (see for example \cite{KP} and \cite{FMY}), the 
\cite[Proposition~4.3]{CFST} shows that $A_{\mathbb{C}}(G_{\Lambda})$ is isomorphic 
to the Kumjian-Pask algebra $\KP_{\mathbb{C}}(\Lambda)$ as defined in \cite{ACaHR}. 
An argument similar to that of the preceding example generalises this to the 
Kumjian-Pask $R$-algebras associated to a locally convex row-finite $k$-graphs 
(possibly with sources) as in \cite{CFaH}.  That is $\KP_{R}(\Lambda) \cong 
A_{R}(G_{\Lambda})$.
\end{rmk}

\section{Groupoid equivalence}

In this section, we assume throughout that $G$ is
a locally compact Hausdorff groupoid and $X$ is a locally compact Hausdorff space.
We say $G$ \emph{acts on the left} of $X$ if there is a map $r_X$ from $X$
onto $\go$ and a map  $(\gamma,x) \mapsto \gamma \cdot x$ from
\[
G*X:= \{(\gamma,x) \in G \times X : s(\gamma)=r_X(x) \} \text{ to } X
\]
such that
\begin{enumerate}
\item if $(\eta, x) \in G * X$ and $(\gamma,\eta)$ is a composable pair in $G$, 
	then  $(\gamma\eta, x),(\gamma, \eta \cdot x) \in G * X$ and 
	$\gamma \cdot(\eta \cdot x) = (\gamma \eta) \cdot x$;
\item $r_X(x) \cdot x = x$ for all $x \in X$.
\end{enumerate}
We will call
$X$ a \emph{continuous left $G$-space} if $r_X$ is an open map and both $r_X$ and $(\gamma,x)
\mapsto \gamma \cdot x$ are continuous.

The action of $G$ on $X$ is \emph{free} if $\gamma \cdot x = x$ implies $\gamma = r_X(x)$.
It is \emph{proper} if the map from $G*X \to X \times X$ given by $(\gamma, x) \to
(\gamma \cdot x, x)$ is a proper map in the sense that inverse images of compact sets
are compact.

We define right actions similarly, writing $s_X$ for the map from $X$ onto $\go$, and
\[
    X*G:= \{(x, \gamma) \in X \times G : s_X(x) = r(\gamma)\}.
\]

\begin{dfn}
\label{def:gheq}
Let $G$ and $H$ be locally compact Hausdorff groupoids.  A \emph{$(G,H)$-equivalence} is a locally compact Hausdorff
space $Z$ such that
\begin{enumerate}
 \item \label{it:gheq1} $Z$ is a free and proper left $G$-space;
\item \label{it:gheq2} $Z$ is a free and proper right $H$-space;
\item \label{it:gheq3} the actions of $G$ and $H$ on $Z$ commute;
\item\label{it:gheq4} $r_Z$ induces a homeomorphism of $Z/H$ onto $\go$;
\item \label{it:gheq5} $s_Z$ induces a homeomorphism of $G\setminus Z$ onto $\ho$.
\end{enumerate}
\end{dfn}

Suppose that $Z$ is a $(G,H)$-equivalence, and that $y, z, y', z' \in Z$ satisfy $s_Z(y) = r_Z(z)$ and
$s_Z(z') = r_Z(y')$. We write ${_G[y,z]} \in G$ and
$[y',z']_H \in H$ for the unique elements such that
\begin{equation}\label{eq:bracket}
{_G[y,z]} \cdot z = y \text{ and } y' \cdot [y',z']_H = z'.
\end{equation}
Let
\[
    Z^{\op}:= \{\overline{z}: z \in Z\}
\]
denote a homeomorphic copy of $Z$. For $z \in Z$, define $r_{Z^{\op}}(\overline{z}) = s_Z(z) \in \ho$
and $s_{Z^{\op}}(\overline{z})= r_Z(z) \in \go$, and for $\eta \in H$ with $s(\eta) =
r_{Z^{\op}}(\overline{z})$ and $\gamma \in G$ with $r(\gamma) = s_{Z^{\op}}(\overline{z})$  define
\[
\eta \cdot \overline{z} := \overline{z \cdot \eta^{-1}}\quad
    \text{ and }\quad
    \overline{z} \cdot \gamma := \overline{\gamma^{-1} \cdot z}.
\]
With this structure, $Z^{\op}$ is an $(H,G)$-equivalence. See \cite{G, MRW87, SW}
for more information on groupoid actions and equivalences.

\begin{rmk}
 Note that if $S$ and $T$ are strongly Morita equivalent inverse semigroups as in
 \cite[Definition~2.1]{Steinberg2},
 then their respective universal groupoids are equivalent \cite[Theorem~4.7]{Steinberg2}.
\end{rmk}

\subsection*{The linking groupoid}
Now suppose that $G$ and $H$ are Hausdorff ample groupoids  and let $Z$ be a $(G,H)$-equivalence. We show that
$A_R(G)$ and $A_R(H)$ are Morita equivalent by embedding them as complementary corners of
the Steinberg algebra of a \emph{linking groupoid} $L$ defined below. In the remainder of
this section, we verify that the linking groupoid in this situation is also a  Hausdorff 
ample groupoid and then
show how $A_R(G)$ and $A_R(H)$ embed into $A_R(L)$.

If $Z$ is a $(G,H)$-equivalence, the \emph{linking groupoid of $Z$} is
defined in \cite[Lemma~2.1]{SW} as
\[
L:= G \sqcup Z \sqcup Z^{\op} \sqcup H,
\]
with $r,s:L \to \lo := \go \sqcup \ho$ inherited from the range and source maps on each
of $G, H, Z$ and $Z^{\op}$.  We write $r$ and $s$ (no subscripts)
 to denote the range and source maps in $L$.  Multiplication $(k,l) \mapsto kl$ in $L$ is given by
\begin{itemize}
\item multiplication in $G$ and $H$ when $(k,l)$ is a composable pair in $G$ or $H$; 
\item  $kl = k \cdot l$ when $(k,l) \in Z * H \sqcup G * Z \sqcup H * Z^{\op} \sqcup Z^{\op} * G;$ and
\item $kl = {_G[k,h]}$ if $k \in Z$ and $l = \overline{h} \in Z^{\op}$, and $kl = [h,l]_H$ if $l \in Z$
	and $k = \overline{h} \in Z^{\op}$.
\end{itemize}
The inverse map is the usual inverse map in each of $G$ and $H$ and is given by $z
\mapsto \overline{z}$ on $Z$ and $\overline{z} \mapsto z$ in $Z^{\op}$.
Both $G$ and $H$ are clopen in $L$ by construction.

\begin{lemma}
Let $G$ and $H$ be Hausdorff ample groupoids. Suppose that $Z$ is a $(G,H)$-equivalence and $L$ is the
linking groupoid of $Z$.  Then $L$ is a  Hausdorff ample groupoid.
\end{lemma}
\begin{proof}
Lemma~2.1 of \cite{SW} implies that $L$ is locally compact and Hausdorff.  
It suffices to show that $L$ is \'etale with totally disconnected unit space.
We have $\lo =\go \sqcup \ho$ which is totally disconnected because $\go$ and $\ho$ are, 
so it remains to show that $L$ is
\'etale.\footnote{If $G$ and $H$ were second-countable, then $L$ would be as well, and
then we could deduce from \cite[Lemma~I.2.7 and Proposition~I.2.8]{Renault} that $L$ is 
\'etale by observing that $\lo$ is open in $L$ (because
each of $\go$ and $\ho$ is open), and the Haar system on $L$ induced from those on $G$
and $H$ consists of counting measures because the systems on $G$ and $H$ have this
property.}

We suppose that $r$ is not a local homeomorphism, and seek a contradiction. Then there exists $z
\in L$ such that $r$ fails to be injective on every neighbourhood of $z$.  Because
$G$ and $H$ are \'etale, $z$ is either in $Z$ or $Z^{op}$.  Without loss of generality, 
assume $z \in Z$; the case for $Z^{op}$ is symmetric. By
choosing a neighbourhood base $\{U_{\alpha}\}$ at $z$ inside of $Z$, we can find a net
$\{(x_\alpha, y_\alpha)\}$  where each $x_{\alpha}, y_{\alpha} \in  U_{\alpha}$ such that:
\begin{enumerate}
\item $x_\alpha, y_\alpha \to z$;
\item \label{it:special2} $x_\alpha \neq y_\alpha$ for all $n$;
\item $r(x_\alpha) = r(y_\alpha)$ for all $n$.
\end{enumerate}
Since $G$ is \'etale, $\go$ is
open in $L$ and so we can assume that $r(x_\alpha) \in \go$ for all
$\alpha$. For each $\alpha$, let
$\gamma_\alpha := [x_\alpha, y_\alpha]_H$, so that $x_\alpha \cdot \gamma_\alpha =
y_\alpha$ for all $\alpha$. Note that $r(\gamma_{\alpha}) = s(x_{\alpha})$.  
Proposition~1.15 of \cite{tf2b} applied to the open map $r : H \to \ho$
implies that, by passing to a subnet, we may assume that $\gamma_{\alpha} \to \gamma \in
H$. So the continuity of the action gives
\[
     z \cdot \gamma = \lim  x_{\alpha} \cdot \gamma_{\alpha} = \lim y_{\alpha} = z.
\]
Since $H$ acts freely on $Z$, this forces $\gamma = s(z)$. Since $\ho$ is
open in $H$, we have $\gamma_{\alpha} \in \ho$ eventually. Hence $x_{\alpha} =
y_{\alpha}$ eventually, contradicting~(\ref{it:special2}).
\end{proof}

Following \cite[page 108]{SW}, for each $F \in A_R(L)$, define $F_{11} = F|_{G}$, $F_{12}
= F|_{Z}$, $F_{21} = F|_{Z^{\op}}$ and $F_{22} = F|_{H}$.
We may view each $F_{ij}$ as an element of $A_R(L)$.  We express the decomposition $F =
\sum_{i,j} F_{ij}$ by writing
\[
 F = \left(\begin{matrix}
      F_{11} & F_{12} \\
F_{21} & F_{22}
     \end{matrix}\right).
\]
It is straightforward to check that convolution in $A_R(L)$ is given by matrix
multiplication for functions written in this form. 
Using this notation, we see that the inclusion maps \[f \mapsto
\left(\begin{matrix}
                       f&0\\0&0
                      \end{matrix}\right)
\text{   and  }
     g \mapsto \left(\begin{matrix}
                       0&0\\0&g
                      \end{matrix}\right)
  \]
define injective homomorphisms $A_R(G) \hookrightarrow A_R(L)$ and $A_R(H)
\hookrightarrow A_R(L)$.  We denote the images of these maps by $i(A_R(G))$ and
$i(A_R(H))$. So
\begin{equation}
\label{eq:iso}
 i(A_R(G)) \cong A_R(G) \text{ and } i(A_R(H)) \cong A_R(G).
\end{equation}

\section{Main result}

\label{sec:main}

We now have the machinery we need to show that equivalent groupoids give rise to Morita
equivalent Steinberg algebras. First, we give the definition of Morita equivalent rings.
Let $A$ and $B$ be rings, $M$ an $A$--$B$ bimodule, $N$ a $B$--$A$ bimodule, and
\[
     \psi:M \otimes_B N \to A \text{ and }\phi:N \otimes_A M \to B
\]
bimodule homomorphisms such that
\begin{equation}\label{eq-ops}
    n'\cdot \psi(m \otimes n) = \phi(n' \otimes m)\cdot n
        \text { and } m'\cdot \phi(n \otimes m) = \psi(m' \otimes n)\cdot m
\end{equation}
for  $n,n' \in N$ and $m,m' \in M$. Then  $(A,B,M,N,\psi, \phi)$ is a \emph{Morita
context} between $A$ and $B$; it is called \emph{surjective} if $\psi$ and $\phi$ are
surjective and in this case we say $A$ and $B$ are \emph{Morita equivalent}.  (See
\cite[page~41]{GS}.)

\begin{theorem}
\label{thm:lme} Let $G$ and $H$ be  Hausdorff ample groupoids.
 Suppose that $Z$ is a $(G,H)$-equivalence with linking
groupoid $L$. Let $i$ denote the inclusion maps from $A_R(G)$ and $A_R(H)$ into $A_R(L)$.  Define
\[
 M:= \{f \in A_R(L) \mid \supp f \subseteq Z\}
    \quad\text{ and }\quad
    N:= \{f \in A_R(L) \mid \supp f \subseteq Z^{\op}\},
 \]
and let $A_R(G)$ and $A_R(H)$ act on the right and left of $M$ and on the left and right
of $N$ by $a \cdot f = i(a)*f$ and $f\cdot a = f*i(a)$. Then there are bimodule
homomorphisms
\[
\psi: M \otimes_{i(A_R(H))}N \to A_R(G) \quad\text{ and }\quad
    \phi: N \otimes_{i(A_R(G))}M \to A_R(H)
\]
determined by
\[
i(\psi(f \otimes g)) = f*g \quad\text{ and }\quad
    i(\psi(g \otimes f)) = g*f.
\]
The tuple $(A_R(G), A_R(H), M, N, \psi, \phi)$ is a surjective Morita context, and so
$A_R(G)$ and $A_R(H)$ are Morita equivalent.
\end{theorem}
\begin{proof}
That $M$ is an $A_R(G)$--$A_R(H)$ bimodule and $N$ is an $A_R(H)$--$A_R(G)$ bimodule is
clear. The given formulas for $\phi$ and $\psi$ are well-defined on the balanced tensor
products because, for example, \[f * (a \cdot g) = f *(i(a)*g) = (f *i(a))*g = (f\cdot a) *
g.\]  The maps $\psi$ and $\phi$ are module homomorphisms by linearity of convolution. The
formula~\eqref{eq-ops} follows from associativity of convolution in $A_R(L)$.

To see that $\psi$ is surjective, it suffices to fix a compact open bisection $U
\subseteq G$ and show that $i(1_U)$ is in the image of $\psi$. For each $x \in r(U)$, choose
$z_x \in Z$ such that $r(z_x) = x$. Since $L$ is \'etale and $Z$ is topologically disjoint
from $G$, each $z_x$ has a neighbourhood $U_x \subseteq Z$ which is a bisection of $L$.
Since $\go$ is locally compact, Hausdorff and totally disconnected, each $x$ has a
compact open neighbourhood $W_x$ contained in $r(U) \cap r(U_x)$, and so by replacing
each $U_x$ with $U_x \cap r^{-1}(W_x)$, we can assume that each $U_x$ is compact open
with $r(U_x) \subseteq r(U)$. Since $r(U)$ is compact, there is a finite set $\{x_1,
\dots, x_n\} \subseteq r(U)$ such that $\bigcup_i r(U_{x_i}) = r(U)$. Let $V_1 = U_{x_1}$ and
iteratively define $V_i = U_{x_i} \setminus r^{-1}\big(\bigcup_{j < i} r(U_{x_j})\big)$. Then the
$V_i$ are compact open subsets of $Z$ on which $r$ and $s$ are bijective, and $r(U)$ is
the disjoint union of the $r(V_i)$. Therefore, writing $V_i^{\op}$ for $\{\overline{z} :
z \in V_i\} \subseteq Z^{\op}$, we have
\[
\left(\begin{matrix}
    1_U&0\\0&0
\end{matrix}\right)
    = \sum_i \left(\begin{matrix}
        0&1_{V_i}\\0&0
    \end{matrix}\right)
    \left(\begin{matrix}
        0&0\\1_{V_i^{\op}}&0
    \end{matrix}\right).
\]
Thus $1_U = \psi(\sum_i 1_{V_i} \otimes 1_{V_i^{\op}})$, and so $\psi$ is surjective. A
similar argument shows that $\phi$ is surjective.

It follows that $(A_R(G), A_R(H), M, N, \psi, \phi)$ is a surjective Morita context, and
so $A_R(G)$ and $A_R(H)$ are Morita equivalent.
\end{proof}


\section{Applications to graph algebras}

\label{SectionGraphApp}
Our aim is to apply our main result to graph algebras. 
First we consider a useful class of
examples of groupoid equivalences --- those arising from \emph{abstract transversals}
of groupoids. Suppose that $G$ is a subgroupoid\footnote{By \emph{subgroupoid} we 
mean a subset that is itself a groupoid.} of $H$ and let $Z:=\go H$. It is straight-forward 
to check that $Z$ is a free and proper left $G$-space and a free and proper right $H$-space 
where $r_Z$ and $s_Z$ are the range and source from $H$ restricted to $Z$ 
and the action is by multiplication in $H$.
Because groupoid multiplication is associative, the actions of $G$ and $H$ commute.
However, $Z$ may not satisfy the surjectivity hypothesis of Definition~\ref{def:gheq}~(5)
required in a groupoid equivalence. The following lemma is a straightforward application of
\cite[Example~2.7]{MRW87}; we give a short proof because the construction is fundamental
to our application of groupoid equivalence to graph algebras.

\begin{lemma}
\label{lem:MRWex} Suppose $H$ is an \'etale groupoid and $X \subseteq \ho$ is a clopen
subset that meets each orbit in $H$. Then $G := XHX$ is a clopen subgroupoid of $H$, and
$Z:= X H$ is a $(G,H)$-equivalence.
\end{lemma}
\begin{proof}
The set $XHX = r^{-1}(X) \cap s^{-1}(X)$ is clopen because $r$ and $s$ are continuous,
and it is clearly a subgroupoid. Similarly, $Z$ is a clopen subset of $H$, and so the
open subsets of $Z$ are the subsets of $Z$ which are open in $H$. Since $H$ is \'etale,
$r$ and $s$ are open maps and so $r_Z$ and $s_Z$ (which are $r$ and $s$ restricted to $Z$) 
are also open maps. The map $r_Z : Z \to X$ is
surjective by definition. To see that $s_Z : Z \to \ho$ is surjective, fix $u \in \ho$. By
hypothesis, $[u] \cap X \not= \emptyset$, so there exists $\alpha \in H$ such that
$r(\alpha) \in X$ and $u = s(\alpha)$. So $\alpha \in Z$ and $u = s(\alpha) \in s_Z(Z)$. 

We prove that
$\tilde{s} : G \backslash Z \to \ho$ is a homeomorphism; the argument that $\tilde{r}$ is
a homeomorphism is similar. Clearly, $\tilde{s}$ is a surjection.
If $\tilde{s}([\alpha]) = \tilde{s}([\beta])$, then
$s(\alpha) = s(\beta)$, and so $\alpha \beta^{-1} \in XHX = G$ and satisfies
$(\alpha\beta^{-1})\cdot \beta = \alpha$.   So $[\alpha] = [\beta]$, and $\tilde{s}$ is
injective. 

To see that $\tilde{s}$ is continuous, suppose $U \subseteq \ho$ is open. Then $HU$ is 
open because $s$ is continuous, and then $ZU = HU \cap Z$ is open in $Z$.  Thus 
$\tilde{s}^{-1}(U) =G\backslash (ZU)$ is open by definition of the quotient topology. 

Finally, if $W \subseteq
G\backslash Z$ is open, then $W = G\backslash W'$ for some open $W' \subseteq Z$. Since
$Z$ is open in $H$, so is $W'$ and then $\tilde{s}(W) = s(W')$ is open because $s$ is
open.
\end{proof}

Given a graph $E$, Crisp and Gow identify a type of subgraph $T$ which can be
``collapsed" to yield a new graph $F$ whose $\cs$-algebra is Morita equivalent to that of
$E$ \cite{CG}. We will demonstrate that $G_E$ and $G_F$ are equivalent groupoids. Bates
and Pask's ``outsplitting'' move described in \cite[Theorem~4.5 and Corollary 5.4]{BP} is
a special case of the Crisp-Gow construction (see \cite[Example~iii]{CG}), as are
S{\o}rensen's moves (S)~and~(R) (see \cite[Propositions~3.1 and 3.2]{S}). So our result
implies that applications of these moves yield Morita equivalent Leavitt path
algebras regardless of the base ring.

When $E$ is countable, our statement of the next proposition corresponds exactly to the 
construction of \cite[Theorem~3.1]{CG} modulo the difference in edge-direction conventions.  
First, we need a few more graph preliminaries.  Suppose $E$ is a directed graph.  For $v \in
E^0$ and $S \subseteq E^0$, we write $v \geq S$ if $S E^* v \not= \emptyset$. We define
the \emph{pointed groupoid} with respect to $S$ to be the subgroupoid of $G_E$ consisting
of groupoid elements $(\alpha x, |\alpha| - |\beta|, \beta x)$ such that $r_E(\alpha), r_E(\beta)
\in S$. We define 
\[E^0_{\sing}:=
\{v \in E^0 : r_E^{-1}(v)\text{ is either empty or infinite}\}.
\]  For $n \in \NN$ we define a map $\sigma^n : \{x \in E^* \cup E^{\infty} : |x| \ge n\} \to E^*\cup E^{\infty}$ 
by $\sigma^n(\alpha
y) = y$ for all $\alpha \in E^n$ (paths of length $n$) and $y \in  E^* \cup E^{\infty}$.  Notice that $\go_E$ is invariant
under $\sigma^n$.
Finally, we say an acyclic path $x \in E^{\infty}$ is a \emph{head} if each $r_E(x_i)$  receives only
$x_i$ and each $s_E(x_i)$ emits only $x_i$.
\begin{prop}\label{prop:cg}
Let $E$ be a directed graph with no heads and suppose that $F^0 \subseteq E^0$ satisfies $E^0_{\sing} \subseteq
F^0$. Suppose also that the subgraph $T$ of $E$ defined by $T^0:= E^0 \setminus F^0$ and
\[T^1:=\{e \in E^1:r_E(e),s_E(e) \in T^0\}\]
is acyclic and that each of the following are satisfied:
\begin{itemize}
\item[(T1)] each vertex in $F^0$ is the range of at most one $y \in E^{\infty}$ such
    that $s_E(y_i) \in T^0$ for all $i \geq 1$;
\end{itemize}
and for each $x \in T^{\infty}$,
\begin{itemize}
\item[(T2)] $r_E(x)\geq F^0$
\item[(T3)] $|s_E^{-1}(r_E(x_i))|=1$ for all $i$; and
\item[(T4)] whenever $s_E(e) = r_E(x)$, we have $|r_E^{-1}(r_E(e))| < \infty$.
\end{itemize}
Let $F$ be the graph with vertex set $F^0$ and one edge $e_{\beta}$ for each path $\beta
\in E^* \setminus E^0$ with $s_E(\beta),r_E(\beta) \in F^0$ and $r_E(\beta_i) \in T^0$
for $1 \leq i < |\beta|$ such that $s_F(e_{\beta}) = s_E(\beta)$ and
$r_F(e_{\beta})=r_E(\beta)$. Let $G \subseteq G_E$ denote the pointed groupoid with
respect to $F^0$. Then
\begin{enumerate}
\item \label{it:cg1} $G$ and $G_E$ are equivalent groupoids and
\item\label{it:cg2} $G$ is isomorphic to $G_F$.
\end{enumerate}
\end{prop}

\begin{rmk}
We will be using \cite[Lemma~3.3]{CG}, which says that if a graph $E$ 
has no heads, satisfies (T1),
(T2) and (T3), and $T$ and $F$ are as above,
then $F^0 \geq v$ for all $v \in T^0$.  Note that this Lemma
also implies that  $r_E^{-1}(v) = \emptyset$ if and only if $r_F^{-1}(v) =
\emptyset$.
\end{rmk}

\begin{proof}
To prove~(\ref{it:cg1}), we will apply Lemma~\ref{lem:MRWex} with $X = \go = F^0
E^\infty$. First notice that
\[
\go = \bigcup_{v \in F^0} Z(v) = \go_E \setminus \Big(\bigcup_{w \in T^0} Z(w)\Big).
\]
Since each $Z(v)$ is open, we deduce that $\go$ is clopen in $\go_E$. Now consider $x \in
\go_E \setminus \go$. We must show that $[x] \cap \go \not= \emptyset$. Since $x \notin
\go$, $r_E(x) \in T^0$. We consider 2 cases.
For the first case, suppose that $\sigma^n(x) \in T^{\infty}$ for some $n$. Then (T2) implies that
there exists $\mu \in E^*$ such that $s_E(\mu) = r_E(x_{n+1})$ and $r_E(\mu) \in F^0$. So
$\mu(\sigma^n(x)) \in [x] \cap \go$.
For the second case, suppose that $\sigma^n(x) \not\in T^\infty$ for all $n$. Since $E^0_{\sing}
\subseteq F^0$, there exists $n$ such that $s_E(x_n) \in F^0$. Hence $\sigma^n(x) \in [x]
\cap \go$.
Now Lemma~\ref{lem:MRWex} implies that $X G_E$ is a $(G, G_E)$-equivalence.

To prove~(\ref{it:cg2}), we first define a map $\phi:\go_F \to \go$, which will
take a little preparation. By construction, $F^1$ is a subset of $E^*$; we write
$\phi_{\fin} : F^1 \to E^*$ for the inclusion map. Since $\phi_{\fin}$ preserves ranges
and sources, we can extend $\phi_{\fin}$ to an injection from $F^*$ to $E^*$ by
\[\phi_{\fin}(\mu) = \phi_{\fin}(\mu_1)\phi_{\fin}(\mu_2) \dots \phi_{\fin}(\mu_{|\mu|}).\]
Again by construction of $F$, we have \[\phi_{\fin}(F^*) = \{\mu \in E^* : r_E(\mu), s_E(\mu)
\in F^0\}.\] We claim that if $v \in F^0$ satisfies $|r_F^{-1}(v)| = \infty$ but
$|r_E^{-1}(v)| < \infty$, then there is a unique infinite path $y_v \in T^\infty$ with
$r_E(y_v) = v$. Indeed, the set
\begin{equation}\label{eq:B}
B_{v} := \{\beta \in E^* \setminus E^0 \mid r_E(\beta)=v, s_E(\beta) \in F^0 \text{ and }
    r_E(\beta_i) \in T^0 \text{ for } 1 \leq i \leq |\beta|\}
\end{equation}
is infinite, and so \cite[Lemma~3.4(d)]{CG} gives such a $y_v$.
That there is a unique such path follows from~(T1).

Define $\phi : \go_F \to \go$ by
\[
\phi(x) = \begin{cases}
    \phi_{\fin}(x_1)\phi_{\fin}(x_2) \dots &\text{ if $x \in F^\infty$;}\\
    \phi_{\fin}(x) &\text{ if $x \in F^*$ and $s_F(x) \in E^0_{\sing}$; and} \\
    \phi_{\fin}(x)y_{s_F(x)} &\text{ if $x \in F^*$, $|r_F^{-1}(s_F(x))| = \infty$,
        and $0 < |r_E^{-1}(s_F(x))| < \infty$.}
\end{cases}
\]
To see that this defines $\phi$ on all $\go_F$ observe that if $x \in \go_F$ belongs to
$F^*$ and $s_F(x) \not\in E^0_{\sing}$, then we have $s_F(x) \in F^0_{\sing} \setminus
E^0_{\sing}$, and since $r_E^{-1}(v) = \emptyset$ if and only if $r_F^{-1}(v) =
\emptyset$, we deduce that $|r_F^{-1}(s_F(x))| = \infty$ and $0 < |r_E^{-1}(s_F(x))| < \infty$.

Since $\phi_{\fin}$ is injective, $\phi$ is also injective. We have 
\begin{align*}
&\phi(F^\infty) = \{x
\in F^0 E^\infty \mid s_E(x_n) \in F^0\text{ for infinitely many $n$}\} \text{ and } \\
&\phi(\{\mu \in F^* : s_F(x) \in E^0_{\sing}\}) = \{\mu \in F^0 E^* : s_E(\mu) \in E^0_{\sing}\}
\end{align*}
 because
$E^0_{\sing} \subseteq F^0$. The complement of these two sets in $\go$ is
\begin{align}
\{x \in F^0 E^\infty &{}\mid s_E(x_i) \not\in F^0\text{ eventually}\} \nonumber\\
    &= \{x \in F^0 E^\infty \mid s_E(x_i) \in T^0 \text{ eventually}\} \nonumber\\
    &= \{\mu y \mid \mu\in F^0 E^* F^0, y \in s_E(\mu) E^\infty, \sigma^1(y) \in T^\infty\}.\label{eq:complement}
\end{align}
Let $\mu y$ be an element of the set~\eqref{eq:complement}.
To see that $\phi$ is surjective, it suffices to show that 
$|r_F^{-1}(r_E(y))| = \infty$, and $0 < |r_E^{-1}(r_E(y))| < \infty$.  For then
$\phi(\phi_{\fin}^{-1}(\mu)) = \mu y$.
Condition~(T4) applied to $e=y_1$ implies that $r_E(y_1)$ is not an infinite receiver in
$E$. We must now show that
$r_F^{-1}(r_E(y_1))$ is infinite. Since $T$ is acyclic, $y$ has no repeating edges or
vertices. Lemma~3.3 of \cite{CG} yields a path $\mu^1 \in E^*$ with $r_E(\mu^1) =
s_E(y_1)$ and $s_E(\mu^1) = v_1 \in F^0$. Since $s_E(\mu^1) \in F^0$, (T3) implies that 
there exists $m_1 < |\mu^1|$ such that $y_j \not\in \{\mu^1_{m_1}, \dots, \mu^1_{|\mu^1|}\}$ 
for all $j$.

Repeating this process for each $n \in \NN$, we obtain distinct paths $\mu^n$ such that
$r_E(\mu^n) = s_E(y_{k_n})$ where $k_n = \sum_{i=1}^{n}(m_i+2)$ and $s_E(\mu^n) \in F^0$.
Now $y_1\dots y_{k_n}\mu^n \in r_F^{-1}(r_E(y))$ for all $n$, and these are distinct
elements of $F^1$, so that $r_F^{-1}(r(y))$ is infinite as required. Therefore, $\phi$ is
surjective. Notice that $\phi$ also preserves concatenation of paths.

Next we show that $\phi$ is a homeomorphism. It takes cylinder sets $Z(\mu)$ in $\go_F$ onto
cylinder sets $Z(\phi_{\fin}(\mu))$ of $\go$, and since it is bijective, it is therefore
open.

To see that $\phi$ is continuous, suppose $x^n \to x$ in $\go_F$.  We consider
the three possibilities for $x$. First, if $x \in F^{\infty}$, then the collection
$\{Z(x_1), Z(x_1x_2), \dots\}$ is a neighbourhood base at $x$ and the collection
\[
    \{\phi(Z(x_1)), \phi(Z(x_1x_2)), \dots\} = \{Z(\phi(x_1)), Z(\phi(x_1x_2)), \dots\}
\]
is a neighbourhood base for $\phi(x)$. So $\phi(x^n)$ converges to $\phi(x)$.

Second, if $x \in F^*$ and $s_F(x)$ is a source, then $\{x\}$ is open in $\go_F$ and hence
$x^n=x$ eventually. Therefore $\phi(x^n)=\phi(x)$ eventually and hence $\phi(x^n)$
converges to $\phi(x)$.

Finally, suppose $x \in F^*$ and $s(x)$ is an infinite receiver. If $x^n$ is eventually
constant then $\phi(x^n)$ converges to $\phi(x)$ as above. So suppose otherwise.  Since
$x^n \in Z(x)$ eventually, we may assume that each
$x^n=xz^n$ where $z^n \in \go_F$.  Also, we have that $\phi(x) = \phi_{\fin}(x) y_{s_E(x)}$.  
Let $B:=Z(\phi_{\fin}(x)y_1 \dots y_m)$ be a basis element containing $\phi(x)$.
Since open sets containing $x$ include sets of the form
\[
 Z(x) \cap \Big(\bigcup_{e \in G} Z(xe)\Big)^c
\]
for finite $G \subseteq r_F^{-1}(s_F(x))$, we may assume that $z^n_1 \not= z^m_1$ for 
$n \not= m$; that is, the first edges of the paths $z^n$ are distinct. Condition~(T4) 
implies that $s_F(x)$ is not an infinite receiver in $E$, so we may also
assume that $\phi(z^n_1) \in E^* \setminus E^1$ for each $n$. So the $\phi(z^n_1)$ are paths in $E$
with range and source in $F^0$ but all other vertices in $T^0$. 
We claim that the distinct paths $\phi(z^n)$  
eventually belong to $Z(y_1y_2 \dots y_m)$. 
Note that \cite[Lemma~3.3]{CG} and (T3) imply
that $|B_{s_E(y_1)}|$ is infinite. Further, for any $e\in
r_E^{-1}(s_F(x)) \setminus \{y_1\}$ we have $|B_{s_E(e)}|<\infty$; for otherwise
\cite[Lemma~3.4(d)]{CG} yields an infinite path that violates~(T1). Hence $\phi(z^n) \in Z(y_1)$
eventually. Similarly, $|B_{s_E(y_2)}|$ is infinite and for any $e\in E^1$ with
$r_E(e)=r_F(y_2)$ we have $|B_{s_E(e)}|<\infty$ so $\phi(z^n) \in Z(y_1y_2)$ for large $n$.
Proceeding in this way, we deduce that for any $m$ we have $\phi(z^n) \in Z(y_1 \dots y_m)$ for
large $n$ as claimed. So $\phi(x(z_n)) \in B$ for large $n$. Thus, $\phi$ is continuous and hence
$\phi$ is a homeomorphism.


Define $\Phi: G_F \to G$ by
\[
\Phi(\mu x,|\mu|-|\nu|,\nu x)
    = (\phi(\mu x),|\phi_{\fin}(\mu)| - |\phi_{\fin}(\nu)|,\phi(\nu x)).
\]
Since $\phi$ preserves concatenation of paths, $\Phi$ is a groupoid homomorphism and it is
straight-forward to show that $\Phi$ is bijective using that $\phi$ is bijective. We have
\[
    \Phi( Z(\mu, \nu) ) = Z(\phi_{\fin}(\mu), \phi_{\fin}(\nu))
\]
for all $\mu,\nu \in F^*$.  So $\Phi$ takes basic open sets in $G_F$ to basic open sets
in $G$, and hence $\Phi$ is open.

To see that $\Phi$ is continuous, suppose $\gamma_n$ converges to $\gamma = (\mu x,k, \nu
x) \in G_F$. So for a basis element
\[
B:= Z(\mu x_1 \dots x_m , \nu x_1 \dots x_m) \cap
    \Big(\bigcup_{\alpha \in F} Z(\mu x_1 \dots x_n \alpha, \nu x_1 \dots x_m \alpha)\Big)^c
\]
containing $\gamma \in G_F$, we eventually have $\gamma_n \in B$. So for large $n$, the element
$\gamma_n$ has the form
\[
    \gamma_n = (\mu x_1 \dots x_m y^n, k, \nu x_1 \dots x_m y^n) \text{ for }y^n \in \go_F.
\]
Thus eventually we have
\[
\Phi(\gamma_n) = (\phi((\mu x_1 \dots x_m y^n),
                    |\phi_{\fin}(\mu)|-|\phi_{\fin}(\nu)|, \phi(\nu x_1 \dots x_m y^n)),
\]
which converges to $(\phi(\mu x), |\phi_{\fin}(\mu)|-|\phi_{\fin}(\nu)|, \phi(\nu x))
= \Phi(\gamma)$.
\end{proof}

\begin{cor}\label{cor:ME}
Suppose $E$ and $F$ are as in Proposition~\ref{prop:cg} and $R$ is a commutative unital ring. Then
\begin{enumerate}
\item\label{it1:corME} $L_R(E)$ is Morita equivalent to $L_R(F)$; and
\item\label{it2:corME} If $E$ is countable, then $\cs(E)$ is Morita equivalent to $\cs(F)$.
\end{enumerate}
\end{cor}
\begin{proof}  Proposition~\ref{prop:cg} implies that $G_E$ and $G_F$ are 
equivalent groupoids.  

Now for~(\ref{it1:corME}), Theorem~\ref{thm:lme} implies that $A_R(G_E)$ and
$A_R(G_F)$ are Morita equivalent, and the result follows from Example~\ref{ex:lpa}.

For~(\ref{it2:corME}), observe that since $E$ is countable, $G_E$ is second countable 
and hence $\cs(G_E)$ is Morita equivalent to $\cs(G_F)$ by \cite[Theorem~2.8]{MRW87}.   
We have $\cs(G_E) \cong \cs(E)$ and $\cs(G_F) \cong \cs(F)$ by \cite[Corollary~3.9]{Pat2002},
and the result follows.  
\end{proof}

\begin{rmk}  Corollary~\ref{cor:ME}(\ref{it1:corME}) generalises
\cite[Proposition~1.11]{ALPS}.
Our proof of Corollary~\ref{cor:ME}(\ref{it2:corME}) provides an alternative proof of
\cite[Theorem~3.1]{CG}.
\end{rmk}

\begin{rmk}S{\o}rensen's move (I) of \cite[Theorem~3.5]{S} is a special case of Bates and Pask's construction
``insplitting'' in \cite[Theorem~5.3]{BP}; a Leavitt path algebra version of this is
proved in \cite[Proposition~1.14]{ALPS}.  In this setting, the corresponding algebras are
actually stably isomorphic.  Both  \cite[Theorem~3.5]{S} and
\cite[Proposition~1.14]{ALPS} can be proved via Steinberg algebras by showing that the
corresponding groupoids are isomorphic.  This was done in the row-finite case by Drinen
in \cite[Proposition~6.1.3]{Drinen}.
\end{rmk}



\end{document}